\documentclass[12pt,a4paper, reqno,dvipsnames]{amsart}

\usepackage{amsfonts}
\usepackage{amsmath}
\usepackage{amssymb,stmaryrd}
\usepackage{amsthm}
\usepackage{amsxtra}
\usepackage{bbm}
\usepackage{braket}
\usepackage{comment}
\usepackage{extarrows}
\usepackage{latexsym}
\usepackage{mathrsfs}
\usepackage{yhmath}
\usepackage{nicematrix}
\usepackage{mathtools}

\usepackage{float}
\usepackage{booktabs}
\usepackage{xcolor}
\usepackage{verbatim}
\usepackage{listings}
\usepackage[normalem]{ulem}
\usepackage[pdfborder={0 0 0}]{hyperref} 
\hypersetup{colorlinks, citecolor=blue, linkcolor=blue, urlcolor=blue}
\hypersetup{pagebackref}

\usepackage{bm}
\usepackage{graphicx}
\usepackage{subcaption}
\usepackage{tikz}
\usepackage{faktor}

\usetikzlibrary{fit}

\AtBeginDocument{%
    \def\MR#1{}
}

\usepackage{cleveref}

\usepackage[cm]{fullpage}

\makeatother

\theoremstyle{plain}
\newtheorem{Theorem}{Theorem}[section]
\newtheorem*{maintheorem}{Main Theorem}
\newtheorem{Lemma}[Theorem]{Lemma}
\newtheorem{Corollary}[Theorem]{Corollary}
\newtheorem{Proposition}[Theorem]{Proposition}

\theoremstyle{definition}

\newtheorem{Assumptions and Discussion}[Theorem]{Assumptions and Discussion}

\newtheorem{Example}[Theorem]{Example}
\newtheorem{Definition}[Theorem]{Definition}

\newtheorem{Problem}[Theorem]{Problem}

\newtheorem{Remark}[Theorem]{Remark}

\theoremstyle{remark}

\newtheorem*{acknowledgment*}{Acknowledgment}

\usepackage[shortlabels]{enumitem}
\SetEnumerateShortLabel{a}{\textup{(\alph*)}}
\SetEnumerateShortLabel{A}{\textup{(\Alph*)}}
\SetEnumerateShortLabel{1}{\textup{(\arabic*)}}
\SetEnumerateShortLabel{i}{\textup{(\roman*)}}
\SetEnumerateShortLabel{I}{\textup{(\Roman*)}}

\def\deg{\operatorname{deg}}

\def\dim{\operatorname{dim}}

\def\floor#1{\left\lfloor #1 \right\rfloor}

\def\Gr{\operatorname{Gr}}

\def\Ht{\operatorname{ht}} 

\def\id{\operatorname{id}}

\def\isom{\cong}

\def\ker{\operatorname{ker}}
\def\KK{{\mathbb K}}

\def\part{\operatorname{part}}

\def\sgn{\operatorname{sgn}}

\newcommand\calF{\mathcal{F}}

\newcommand\calH{\mathcal{H}}

\newcommand\calJ{\mathcal{J}}
\newcommand\calK{\mathcal{K}}
\newcommand\calL{\mathcal{L}}
\newcommand\calM{\mathcal{M}}

\newcommand\calP{\mathcal{P}}

\newcommand\calR{\mathcal{R}}

\newcommand{\Lap}{\operatorname{LAP}}

\newcommand{\Proj}{\operatorname{Proj}}
\newcommand{\rank}{\operatorname{rank}}
\newcommand{\rvline}{\hspace*{-\arraycolsep}\vline\hspace*{-\arraycolsep}}
\def\reg{\operatorname{reg}}

\def\aa{{\bf a}}
\def\LL{{\bf L}}
\def\kk{{\bf k}}

\def\bfi{{\bf i}}
\def\jj{{\bf j}}

\def\aa{{\mathbf{a} }}

\begin{document}

\title{Special fibers of coordinate sections of Hankel Matrices}

\author{K. Ansaldi, D. Lira,  M. Mostafazadehfard, K. Saloni, and 
 L.Seccia }



\address{Wabash College}
\email{ansaldik@wabash.edu}
\address{Federal Rural University of the Semi-Arid Region}
\email{dayane.lira@ufersa.edu.br}
\address{Federal University of Rio de Janeiro}
\email{maral@im.ufrj.br}
\address{Indian Institute of Technology Patna}
\email{ksaloni@iitp.ac.in}
\address{University of Neuchâtel}
\email{lisa.seccia@unine.ch}

\maketitle

\begin{abstract}
 We investigate the special fibers associated with certain coordinate sections of Hankel determinantal ideals. We provide explicit descriptions of their defining equations, showing that these equations admit a natural matrix structure. In particular, we prove that they are Cohen–Macaulay and cannot, in general, be minimally generated only by quadrics and cubics. Instead, we show that the degrees of their minimal generators grow with the size of the minors involved. In one case, we also prove that the Rees algebra is of fiber type. Additionally, we compute algebraic invariants of these special fibers. Our results partially build on and extend the work of Ramkumar and Sammartano on $2$-determinantal ideals and answer some of the questions posed by Cunha, Mostafazadehfard, Ramos, and Simis in earlier work. 
\end{abstract}

\section{Introduction}
Throughout this paper, by \textit{generic Hankel matrix} we mean a matrix of the form 

$$\calH_{l,m} \coloneqq \left(\begin{matrix}
    x_1 & x_2 & \cdots & x_{m}\\
    x_2 & x_3 &   \cdots & x_{m+1}\\
    \vdots &  \vdots && \vdots \\
    x_{l}& x_{l+1} &  \cdots & x_{l+m-1}
\end{matrix}
\right ),$$
where $x_1,\ldots,x_{l+m-1}$ are indeterminates over an infinite field $\KK$ of characteristic $0$. When $l=m$, the matrix $\calH_{l,l}$ is a generic symmetric matrix whose entries are constant along each anti-diagonal.\par 
By \textit{Hankel determinantal ideal} we mean the ideal $I_t(\calH)$ in the standard graded polynomial ring $\KK[x_1,\ldots,x_{l+m-1}]$ generated by the $t$-minors of $\calH$. Hankel determinantal ideals have been widely studied for their deep connections with algebraic geometry. Typical geometric objects defined by these ideals are rational normal scrolls and their secant varieties, but they also appear in applications to discrete geometry, graph theory, representation theory, and statistics (see e.g., \cite{AFS,Raicu2013}). \par

In commutative algebra these ideals are well understood thanks to the work of \cite{BV,Co,Ei88}. They define irreducible varieties and their homological properties and algebraic invariants have been determined. The structure and singularities of Hankel determinantal ideals have been further analyzed in \cite{CMRS, CMSV, Sec}. A substantial amount of work has addressed the problem of studying their blowup algebras. Determining the  defining equations of these algebras is a classical and challenging problem in algebra and geometry; it corresponds to finding the implicit equations of graphs and images of rational maps in projective space, and is therefore known as the \textit{implicitization problem}. Within the framework of Hankel determinantal ideals, this problem was first solved by Conca, Herzog, and Valla \cite{CHV} for \textit{balanced} normal scrolls, and later extended by Sammartano \cite{Sam20} to any rational normal scrolls. Building on these results, in \cite{RS} the authors also compute the defining equations of the blowup algebras of arbitrary $2$-determinantal ideals. The case of secant varieties of rational normal scrolls has instead been investigated in \cite{Nam, LS}. 
\medskip 

In this work we focus on blowup algebras of coordinate sections of Hankel determinantal ideals, obtained by setting the last $r$ entries of  $\calH$ to zero (for details, see Section \ref{sec:Hankelsec}). In particular, we consider sub-maximal minors, whose algebro-geometric properties were first studied in \cite{CMRS} and are naturally motivated by their geometric significance and connections to polar maps as explained in \cite{CRS}.\par We analyze the following two extremal degenerations: 
$$
\calH[n-1]\coloneqq\left(
	\begin{matrix}
	x_1&x_2&x_3&\ldots  & x_{n+1}\\
	x_2&x_3& x_4& \ldots & x_{n+2}\\
	\vdots &\vdots & \vdots  &\iddots  &\vdots \\
	x_{n}&x_{n+1}& x_{n+2} & \ldots   &0\\
	x_{n+1}&x_{n+2}& 0 & \ldots  & 0\\
	\end{matrix}
	\right)
\quad 
\calH[1]\coloneqq\left(
\begin{matrix}
x_1&x_2&\ldots &x_{n} & x_{n+1}\\
x_2&x_3&\ldots &x_{n+1}& x_{n+2}\\
\vdots &\vdots &\iddots &\vdots &\vdots\\
x_{n}&x_{n+1}&\ldots &x_{2n-1} &x_{2n}\\
x_{n+1}&x_{n+2}&\ldots &x_{2n}& 0\\
\end{matrix}
\right).
$$

The matrix $\calH[n-1]$ represents the most degenerate case, where the maximum number of entries is set to zero (i.e., the largest number of variables that can be set to zero without making any of the sub-maximal minors vanish). In contrast, the matrix $\calH[1]$ corresponds to the least degenerate case, in which there is exactly one zero in the bottom-right corner. \par
Our main results can be summarized as follows:

\begin{maintheorem}\label{mainthm}
   \begin{enumerate}[topsep=1.5pt,itemsep=1ex]
       \item[\emph{(1)}] Let $I=I_n(\calH[n-1])$ be the ideal of sub-maximal minors of $\calH[n-1]$:
       \begin{itemize}[topsep=1.5pt,itemsep=1ex]
           \item The special fiber ring $\calF(I)$ is Cohen-Macaulay and its defining ideal is generated by the Pl\"ucker relations and the Laplace relations \eqref{eq:lap}. These relations are all quadratic and form a Gr\"obner basis of the defining ideal of the special fiber ring $\calF(I)$. In particular, the special fiber is Koszul. 
           \item Let $J[n-1]$ be the coordinate sections of the gradient ideal of $\det \calH$. The rational map $\Theta_{J[n-1]}$, associated with $J[n-1]$ is a  birational map onto its image. As a consequence, we deduce the multiplicity of the special fiber of $J[n-1]$.
       \end{itemize}
       \item [\emph{(2)}] Let $I=I_n(\calH[1])$ be the ideal of sub-maximal minors of $\calH[1]$:
       \begin{itemize}[topsep=1.5pt,itemsep=1ex]
           \item The special fiber ring $\calF(I)$ is a Gorenstein ring whose defining ideal is minimally generated by the Pl\"ucker relations  and exactly one extra relation of degree $n$, arising from the vanishing of the determinant of the square matrix of size $n^2$ \eqref{eq:polynomial f}.
           \item The Rees algebra $\calR(I)$ is of fiber type; hence its defining equations are given by the relations defining the special fiber, and the Eagon-Northcott relations defining the symmetric algebra. Moreover, $I$ has linear powers although the Rees algebra is not Koszul.
       \end{itemize}
   \end{enumerate}
\end{maintheorem}
These results provide a partial answer to some of the questions asked in \cite[Question 2.9]{CMRS}. In both cases, the special fiber is Cohen–Macaulay; in the second case it is even Gorenstein, and the ideal is of fiber type. However, we show that in general the special fiber is not minimally generated by quadrics and cubics as conjectured in \cite{CMRS}. Instead, the degrees of the minimal generators depend on both the number of zeros and the size of the matrix.

It is worth emphasizing that all the relations we find have a nice matrix structure; that is, they can be derived from the vanishing of certain determinants using generalized Laplace expansion. This is particularly significant, as relations among non-maximal minors are notoriously difficult to describe, especially when higher-degree relations beyond the Plücker quadrics appear, as in our case. \par 
Notice  that all the ideals considered in this work have codimension $3$ but are not Gorenstein. Results on blowup algebras of codimension 3 Gorenstein ideals can be found in \cite{KPU17,Mo96}. In particular, they are all $n$-determinantal ideals in the sense of \cite{RS}; that is, ideals of $n$-minors of expected codimension. For these reasons, we expect that our results may provide some insights toward a  generalization of the results in \cite{RS} from $2$-determinantal ideals to $n$-determinantal ideals of codimension $3$. In this context, the matrix $H[n-1]$, obtained from $\calH[n-1]$ by rearranging its entries into a rectangular $n\times (n+2)$ matrix (see Equation \eqref{Hrecatngular}), might be interpreted as a higher-dimensional analog of the Jordan block with eigenvalue 0 used in \cite{RS} to construct the Kronecker-Weierstrass (KW) normal form of a $2$-determinantal ideal. In proving Main Theorem~(1), we prove in fact that the special fiber of $I_n(\calH[n-1])$ is isomorphic to that of the ideal of 2-minors of

$$ \begin{pmatrix}
    x_1 & x_2 &  x_3 & \cdots & x_{n+2}\\
    x_2 & x_3 &  \cdots & x_{n+2}& 0
\end{pmatrix}.$$

In the notation of \cite{RS}, this is an ideal in $\mathscr{H}_{n+1,0}$; that is, a 2-determinantal ideal of codimension $n+1$ with no scroll blocks and exactly one Jordan block with eigenvalue 0 in its KW normal form. However, we stress that the case $n>2$ is more delicate, since no analog of the KW normal form is known for $n$-determinantal ideals, and it is plausible that such a classification does not exist.\par
Finally, we emphasize that the birationality result in Main Theorem~(1) is particularly notable, as it yields the multiplicity of the special fiber of a coordinate section of the gradient ideal. This leads to a new interpretation of reduction ideals with maximal possible multiplicity, which we refer to as \textit{maximal reductions}. These ideals simultaneously satisfy both properties: they are reduction ideals of $I$, and the rational map associated to them is birational.

\subsection*{Outline of the paper}
We briefly provide a description of the contents of each section. Section \ref{sec:prelim} is devoted to introducing the notation and recalling some of the definitions and known results that will be used throughout the paper. Section \ref{sec:extremZero} focuses on the most degenerate case, corresponding to the matrix $\calH[n-1]$. To compute the defining equations of the special fiber of $I_n(\calH[n-1])$, we reduce to the case of $2$-determinantal ideals studied by Ramkumar and Sammartano in \cite{RS} (Theorem \ref{thm:eqSF}). Inspired by the classical duality property of the Grassmannian, we achieve this reduction using the 
$\KK$-algebra isomorphism \eqref{eq:Kisom} induced by the isomorphism between the posets of minors of the corresponding Grassmannians (see Remark \ref{rmk:isoposet}). As a consequence we deduce that the special fiber has a Gr\"obner basis generated by quadrics, and so it is Koszul (Theorem \ref{cor:SFKoszul}). We conclude by studying the birationality of the map associated with a special reduction of the ideal $I_n(H[n-1])$ (Section \ref{sec:geometry}).
Section \ref{sec:r=1} addresses the opposite situation, where exactly one variable is set to zero. We show that in this case the special fiber is Gorenstein and is generated by the Plücker relations together with exactly one additional relation of degree $n$, which arises from the vanishing of an  $n^2$-minor in the variables $x_i$ (Theorem \ref{thm:eqSFr=1}). Moreover, we prove that the ideal has linear powers and is of fiber type, so we have an explicit description of the equations of the Rees algebra (Proposition \ref{prop:r=1 fiber type}).  In addition to the previous results, we compute in both cases algebraic invariants such as regularity, multiplicity, reduction number and a-invariant (Corollaries \ref{mult}, \ref{cor:alginvariant r=n-1}, \ref{cor:alginvariant r=1}, \ref{cor: mult}).

\section{Preliminaries and notations}\label{sec:prelim}

\subsection{Coordinate sections of Hankel matrices}\label{sec:Hankelsec}
Throughout this paper, we consider generic square \emph{Hankel matrices} of the form 
$$\calH_{n+1,n+1} \coloneqq \left(\begin{matrix}
    x_1 & x_2 &  x_3 & \cdots & x_{n+1}\\
    x_2 & x_3 &  \cdots & \cdots & x_{n+2}\\
    x_3 & \cdots & \cdots& \cdots & x_{n+3}\\
    \vdots & \iddots& \vdots &\iddots& \vdots \\
    x_{n+1}& x_{n+2} & x_{n+3} & \cdots & x_{2n+1}
\end{matrix}
\right ).$$
By \textit{coordinate section} of $\calH$ we mean the matrix obtained from $\calH$ by setting certain variables equal to zero. Geometrically, this corresponds to taking coordinate hyperplane sections of the ambient space, which justifies the terminology. Motivated by \cite{CMRS}, in this paper we consider coordinate sections on the last $r$ variables for $r=0,1,\ldots,n-1$. We  denote the resulting matrix by $\calH[r]$:

$$
\calH[1]\coloneqq\left(
\begin{matrix}
x_1&x_2&\ldots &x_{n} & x_{n+1}\\
x_2&x_3&\ldots &x_{n+1}& x_{n+2}\\
\vdots &\vdots &\iddots &\vdots&\vdots \\
x_{n}&x_{n+1}&\ldots &x_{2n-1} &x_{2n}\\
x_{n+1}&x_{n+2}&\ldots &x_{2n}& 0\\
\end{matrix}
\right),
\;
\calH[2]\coloneqq\left(
\begin{matrix}
x_1&x_2&\ldots &x_{n-1} &x_{n} & x_{n+1}\\
x_2&x_3&\ldots& x_{n} &x_{n+1}& x_{n+2}\\
\vdots &\vdots &\iddots &\vdots &\vdots &\vdots \\
x_{n-1}&x_{n}&\ldots &x_{2n-3} &x_{2n-2} &x_{2n-1}\\
x_{n}&x_{n+1}&\ldots &x_{2n-2} &x_{2n-1} &0\\
x_{n+1}&x_{n+2}&\ldots &x_{2n-1} &0& 0\\
\end{matrix}
\right),
$$

	$$\ldots, 
	\calH[n-1]\coloneqq\left(
	\begin{matrix}
	x_1&x_2&x_3&\ldots &x_{n} & x_{n+1}\\
	x_2&x_3& x_4& \ldots &x_{n+1}& x_{n+2}\\
	x_3&x_4& x_5 &\ldots  &x_{n+2}& 0\\
	\vdots &\vdots & \vdots  &\iddots &\vdots &\vdots \\
	x_{n}&x_{n+1}& x_{n+2} & \ldots  & 0 &0\\
	x_{n+1}&x_{n+2}& 0 & \ldots  &0& 0\\
	\end{matrix}
	\right).
	$$

In this notation, we retrieve the generic case $\calH$ as $\calH[0]$. For a given $r$, the ground ring for $\calH[r]$ is the polynomial ring $S_r=\mathbb{K}[x_1, \ldots,x_{2n+1-r}]$. However, if
no confusion arises, when $r$ is fixed in the discussion, we will denote this ring simply by $S$. \\

We will be interested in studying blowup algebras of the ideals generated by sub-maximal minors of $\calH[r]$. As is customary in the literature on determinantal ideals, we denote them by $I_n(\calH[r])$.

A key structural property about Hankel matrices allows us to reduce the study of sub-maximal minors to that of maximal minors. This was first observed by Gruson–Peskine \rm(\cite{GP82}) for generic Hankel matrices. However, the same holds for coordinate sections of $\calH$ (see \cite[Lemma 1.4]{CMRS} for an explicit statement). 

\begin{Proposition}[\cite{CMRS,GP82}]\label{Prop:maxminors}$I_n(\calH_{n+1,n+1}[r])=I_n(\calH_{n,n+2}[r])$, for all $r=1,\ldots,n-1$. 
\end{Proposition}

      Since it is often more convenient to work with maximal minors, from now on we denote by $H[r]$ the rectangular $n \times (n+2)$ degenerate Hankel matrix, that is 
     \begin{equation}\label{Hrecatngular}
         H[r]:=\calH_{n,n+2}[r]=\left(\begin{matrix}
   x_1 & \ldots&  x_{n+2-r} & \cdots &\cdots & x_{n+2}\\
    x_2 & \ldots &  x_{n+3-r} & \cdots &\cdots & x_{n+3}\\
   \vdots &  & \vdots & &  & \vdots \\
   x_{n-r}& \ldots & x_{2n+2-r} & \cdots &\cdots & x_{2n+1-r}\\
   \vdots&  & \vdots &  & \iddots &0\\
    \vdots & & \vdots &\iddots&  &\vdots \\
    x_{n}& \ldots & x_{2n+1-r} & 0 &\cdots & 0
\end{matrix}
\right ).
     \end{equation}

\medskip

 It is well known that  the ideal  $I_t(\mathcal{H}_{n+1,n+1})$ is prime and has expected codimension $2n-2t+3$ for any $1\leq t\leq n+1$. In particular, for $t=n$, we get $\Ht (I_n(\mathcal{H}_{n+1,n+1}))=3$. The same properties hold for the corresponding coordinate sections; for details, we refer the reader to \cite[Proposition 4.3]{Ei88} and  \cite[Proposition 2.4]{CMRS}. In particular, $I_n(\mathcal{H}[r])$ defines a Cohen-Macaulay ring.

\subsection{Blowup algebras}
The focus of our work is examining blowup algebras associated to the ideal $I_n(H[r])$ where $1\leq r\leq n-1$. For $0\leq r\leq n-1$ the \emph{Rees Algebra} 
and the \textit{special fiber ring} of $I_n(H[r])$ are denoted by $\calR[r]$ and $\calF[r]$, respectively, and are defined as:

$$\calR[r]:=S[I_n(H[r])t]=\bigoplus_{a\geq 0} (I_n(H[r]))^at^a\subset S[t]$$
$$\calF[r]:=\calR[r]\otimes_S \mathbb{K}\cong \mathbb{K}[I_n(H[r])t].$$
Since all generators of $I_n(H[r])$ are homogeneous of the same degree, it follows that $\mathcal{F}[r]\cong \mathbb{K}[I_n(H[r])].$
Both algebras can be realized as quotients of the polynomial rings

$$\calP_{\calR}\coloneqq S[T_{\bfi} |\ \bfi=\{i_1,\ldots, i_n\}\ \text{is an $n$-subset of} \ \{1,\ldots,n+2\}],$$
and
$$\calP_{\calF}\coloneqq \mathbb{K}[T_{\bfi} |~\bfi=\{i_1,\ldots, i_n\} ~\text{is an $n$-subset of}~ \{1,\ldots,n+2\}].$$ By convention, indices are considered up to permutation; for example, $T_{\{1,3,4\}}=T_{\{3,4,1\}}=T_{\{4,1,3\}}$.

Consider the following surjective homomorphisms 

\begin{equation}\label{naturalHom}
\begin{aligned}
\phi: \mathcal{P}_{\mathcal{R}} &\longrightarrow \mathcal{R}[r] &{\rm{ and }} &&
\psi: \mathcal{P}_{\mathcal{F}} &\longrightarrow \mathcal{F}[r] \\
T_{\bfi} &\longmapsto [\bfi]\,t &&&
T_{\bfi} &\longmapsto [\bfi] 
\end{aligned}
\end{equation}

where $[\bfi] $ denotes the maximal minor corresponding to the ordered set of columns $\{i_1, \ldots,i_n\}$.
The kernels of these maps, denoted by
$$\calJ[r]:=\mathrm{ker}(\phi)\quad {\rm{ and }} \quad   \calK[r]:= \mathrm{ker}(\psi),$$
are called the\textit{ defining} or \textit{presentation} \textit{ideals} of the Rees algebra $\calR[r]$ and the special fiber ring $\calF[r]$ respectively. 
\medskip 

The Rees algebra $\calR[r]$ has a natural \emph{bigrading} given by 
\[
\deg(x_i) = (1, 0)
\quad  \quad
\deg(T_{\bfi}) = (0, 1)
\quad  \quad
\deg(t)=(-n,1). \]
 With this assignment, the $\mathbb{K}$-algebra homomorphisms $\phi$ and $\psi$ are homogeneous of degree zero, and the special fiber ring $\calF[r]$ is the homogenous subring $\calF[r]\subset \calR[r]$ concentrated in bidegrees $(0,*)$. 
 The bidegree $(k,l)$ in $\mathcal P_{\mathcal R}$ is defined so that the first component $k$ records the polynomial degree in $S$, while the second component $l$ corresponds to the degree in the variables $T_{i}$ over $S$, or equivalently, to the power of the variable $t$.

The dimension of the special fiber ring of an ideal $I$ is an important invariant, called the \textit{analytic spread} of $I$, and denoted as $\ell(I)=\dim(\calF(I))$. This invariant is closely related to reduction ideals and birationality (see Section \ref{sec:reduction} and Section \ref{sec:geometry}).

\begin{Proposition} \cite[Proposition 1.89]{Vas}\label{Prop1.7}
    Let $I$ be an ideal in a standard graded polynomial ring $S$. The following inequalities hold
$$\mathrm{ht}(I)\leq \ell(I) \leq \min\{\mu(I), \mathrm{dim}(S)\},$$
where $\mu(I)$ is the minimal number of generators of $I$.
We say that $I$ has \text{maximal analytic spread} if $$\ell(I) =\min\{\mu(I), \mathrm{dim}(S)\}.$$
\end{Proposition}
     
\begin{Remark}\label{rmk:maxanspread}
    By \cite[Proposition 2.8]{CMRS} all the ideals considered in this paper have maximal analytic spread, that is, $\ell(I_n(H[r])=\dim S_r=2n-r+1$.
\end{Remark}

Finding the defining equations of $\mathcal{F}[r]$ and $\mathcal{R}[r]$ is a difficult task, but it plays a key role in the study of their homological and geometric properties (see, e.g., \cite{BCV}, \cite{CHV}, \cite{Sam20}). Of special interest is the case when the Rees algebra is defined by the equations of the special fiber, together with those of the symmetric algebra.

\begin{Definition}
  The ideal $I_n(H[r])$ is said to be \emph{of fiber type} if $\mathcal{J}[r]$ is generated in bidegrees $(0,*)$ and $(*,1)$. Equivalently, the defining equations of $\calR[r]$ consist of those of $\calF[r]$ together with those arising from the first syzygies of $I_n(H[r])$.
\end{Definition}

\subsection{Reductions of an ideal}\label{sec:reduction}
It is not our intent to review the theory of reductions of ideals, the interested reader can consult directly \cite{HS}. However, we wish to introduce a
few concepts for the sake of clarity.
\medskip

A reduction of an ideal $I$ is, roughly speaking, a smaller ideal that controls the asymptotic behavior of the powers of $I$.  
\begin{Definition}\cite[Definition 1.2.1 and 8.3.1]{HS}
   Let $I$ be an ideal in a ring $R$. An ideal $J\subset I$ is called a \emph{reduction} of $I$ if $JI^k=I^{k+1}$ for some integer $k$. A reduction $J$ of $I$ is called a \emph{minimal reduction} if it is minimal with respect to inclusion, i.e., if it does not properly contain another reduction of $I$.
\end{Definition}

We now recall the notion of absolute reduction number of an ideal $I$, an algebraic invariant that measures how far a reduction is from generating the powers of $I$. 
\begin{Definition} \cite[Definitions 8.2.3]{HS}
  Let $J$ be a reduction of $I$. The \textit{reduction number} of
$I$ with respect to $J$ is the minimum integer $k$ such that $JI^k=I^{k+1}$. It is denoted by ${\rm r}_J(I)$. \par 
The \emph{absolute reduction number of $I$} is $$\mathrm{min}\{{\rm r}_J(I): J ~~\text{is a minimal reduction of}~~ I\}.$$ 
\end{Definition}

In general, a minimal reduction of an ideal does not always exist, even over a Noetherian ring. However, when working over a Noetherian local ring or a standard graded polynomial ring (as in our case), minimal reductions always exist. Although they are not unique in general, under mild hypotheses, the number of generators of any minimal reduction is the same, and it is equal to the analytic spread of the ideal.

\begin{Proposition}\cite[Proposition 8.3.7]{HS}\label{Prop1.6}
    If $S$ is a standard graded polynomial ring with \textit{infinite} residue field and $J$ is a reduction of $I$, then $\mu(J)=\ell(I)$ if and only if $J$ is a minimal reduction of $I$.
\end{Proposition}

\subsection{Eagon-Northcott and Pl\"ucker relations}
We conclude this section by reviewing classical relations among maximal minors that will be used later, as they arise naturally in the defining equations of the blowup algebras of $I_n(H[r])$. 
\medskip

\subsubsection{Eagon-Northcott relations}  The ideal $I_n(H[r]) $ is perfect of expected codimension, and it is generated by the maximal
minors of a matrix with linear entries. Therefore, it has a linear minimal free resolution given by the Eagon–Northcott complex. Here, we only recall the first linear syzygies arising from this resolution, and we refer the interested reader to \cite[Chapter 2.C]{BV} for further details.
\medskip

Let $i_1, \ldots, i_{n+1} \in [n+2]$ be arbitrary increasing column indices, and let $k\in [n]$. Consider the
$(n+1) \times (n+1)$ matrix

$$\begin{pmatrix}
    C_{i_1}& C_{i_2} &\ldots &C_{i_{n+1}}\\
    x_{i_1+k-1} &x_{i_2+k-1} &  \ldots&  x_{i_{n+1}+k-1}
\end{pmatrix},$$
where $C_{i_j}$ denotes the $i_j$-th column of $H[r]$ and we set $x_{i_j+k-1}=0$ if $i_j+k-1> 2n+1-r$. In other words, the $k$-th row of $\begin{pmatrix}
    C_{i_1}& C_{i_2} &\ldots &C_{i_{n+1}}\
\end{pmatrix}$ is repeated. Since the determinant of this matrix is zero, expanding along the last row yields the linear syzygy

$$\sum \limits_{j=1}^{n+1} x_{i_j+k-1} [i_1 \ldots \widehat{i_j} \ldots i_{n+1}]=0,$$
where $\ \widehat{}\ $ denotes omission. Therefore, the symmetric algebra of $I_n(H[r]) $ is generated by the following relations:

\begin{equation}\label{eq:ENeq}
    \sum \limits_{j=1}^{n+1} x_{i_j+k-1} T_{i_1 \ldots \widehat{i_j} \ldots i_{n+1}}.
\end{equation}

\subsubsection{Pl\"ucker relations}
Let $\Gr(n,n+2)$ denote the Grassmannian of
$n$-dimensional subspaces of an $(n+2)$-dimensional vector space over the base field~$\mathbb{K}$.  The Pl\"ucker embedding realizes
$\Gr(n,n+2)$ as a subvariety of
\(\mathbb{P}^{\binom{n+2}{n}-1}=\Proj\mathbb{K}[T_{\bfi}\mid |\bfi|=n]\).
The \emph{Pl\"ucker ideal}, denoted by $\mathscr{P}_{n,n+2}$, is the ideal of
algebraic relations among the $[{\bfi}]$. Equivalently, it encodes all the relations among the maximal minors of an $n \times (n+2)$ generic matrix. It is well known that the \emph{Pl\"ucker ideal} is a prime homogeneous
ideal generated by quadratic relations known as the
\emph{Pl\"ucker relations}.  For any subsets of indices $\mathsf{I},\mathsf{J} \subset [n+2]$ such that $|\mathsf{I}|=n-1$ and $|\mathsf{J}|=n+1$, the corresponding Pl\"ucker relation is the well known alternating sum
\begin{equation}\label{eq:Pluck}
    \sum_{\lambda=1}^{n+1} (-1)^{\lambda}[i_1,\ldots,i_{n-1}, j_{\lambda}][j_1,\ldots, \hat{j_{\lambda}},\ldots,j_{n+1}]=0,
\end{equation}
where $\ \widehat{}\ $ denotes omission. Note that if $\mathsf{I}\subseteq \mathsf{J}$, then the above relation is trivial, because each term of the sum either vanishes or is canceled by another term with opposite sign. Hence, the only nontrivial relations arise from subsets of the form $\mathsf{I}=\{i_1, \ldots, i_{n-2},a\}$ and $\mathsf{J}=\{i_1, \ldots, i_{n-2},b,c,d\}$ with $a<b<c<d$. The corresponding nontrivial Pl\"ucker relation on maximal minors is

\begin{equation}\label{eq:minors relation}
\begin{aligned}
    [i_1,\ldots,i_{n-2}, a,b][i_1,\ldots,i_{n-2}, c,d]-[i_1,\ldots,i_{n-2}, a,c][i_1,\ldots,i_{n-2}, b,d]+\\
    +[i_1,\ldots,i_{n-2}, a,d][i_1,\ldots,i_{n-2}, b,c]=0.
\end{aligned}
\end{equation}
Note that each relation~\eqref{eq:minors relation} contains exactly
three terms. 

\begin{Remark}\label{nonTrivialPL}
    The Pl\"ucker ideal \(\mathscr{P}_{n,n+2}\) is generated by all the quadrics of the form:
\begin{equation}\label{eq:general-plucker}
  \text{PLU}_\mathsf{I}=  T_{\mathsf{I}\cup\{a,b\}}T_{\mathsf{I}\cup\{c,d\}}
- T_{\mathsf{I}\cup\{a,c\}}T_{\mathsf{I}\cup\{b,d\}}
+ T_{\mathsf{I}\cup\{a,d\}}T_{\mathsf{I}\cup\{b,c\}},
\end{equation}
where $\mathsf{I}\subset [n+2]$ with $|\mathsf{I}|=n-2$ and $a,b,c,d$ are the remaining distinct indices, with $a<b<c<d$. 

Each term in a Pl\"ucker relation is a product of two variables $T_\bfi$ and $T_\jj$ such that index sets $\bfi$ and $\jj$ share exactly $n-2$ elements.  
\end{Remark}

\begin{Remark}\label{rmk:chainF}
        By \cite[Proposition 4.1]{CMRS} the special fiber of $I_n(\calH)$ is isomorphic to the special fiber of the ideal of maximal minors of a generic matrix of the same size, which is well known to be isomorphic to the coordinate ring of the Grassmannian. Consequently, the Pl\"ucker relations generate the defining ideal $\calK[0]$ of the special fiber of $I_n(\calH)$. When we degenerate $\calH$, higher-degree equations may appear. However, we still have the following chain of inclusions:
        
    \begin{equation}\label{containment}
        \mathscr{P}_{n,n+2}=\calK[0]\subseteq \calK[1]\subseteq\calK[2]\subseteq \ldots \subseteq \calK[n-1].
\end{equation}
   
    \end{Remark}

\section{Degeneration $r=n-1$} \label{sec:extremZero}

In this section, we study the extremal degeneration in which the Hankel matrix contains the maximal number of zero entries allowed by the pattern, that is, when $r = n-1$ and

\begin{equation*}
    H[n-1]=
	\begin{pmatrix}
	x_1&x_2&x_3&x_4&\ldots &x_{n+1}&x_{n+2} \\
	x_2&x_3& x_4& x_5&\ldots &x_{n+2}&0\\
	\vdots &\vdots & \vdots & \vdots  &\iddots &\vdots&\vdots \\
    x_{n-1}&x_{n}& x_{n+1}& x_{n+2} &\ldots  & 0&0 \\
	x_{n}&x_{n+1}& x_{n+2}& 0 &\ldots  & 0 &0\\
	\end{pmatrix}.
\end{equation*}
\vspace{3pt}

We begin by determining the defining equations of the special fiber $\mathcal{F}[n-1]$, that is, a set of generators for the ideal $\calK[n-1]$. 

By Remark \ref{rmk:chainF}, the Plücker ideal is contained in $\calK[n-1]$. In this case, however, the Pl\"ucker relations alone do not generate $\calK[n-1]$; additional relations appear. Nevertheless, these extra relations are still quadratic and admit a matrix structure. To describe them precisely, we introduce the following notation.
\medskip

In the matrix $H[n-1]$, let $C_i$ denote the $i$-th column. We can write
$$H[n-1]=\begin{pmatrix}C_1 & C_2 & \cdots & C_{n+1} & C_{n+2}
\end{pmatrix}.$$ 
Let $\mathbf{a}=\{a_1, a_2, \ldots, a_{n-2}\} \in \binom{[n+2]}{n-2}$, where $\binom{[N]}{c}$ denotes the set of all subsets of $[N]$ of cardinality $c$. We define
$$\calL_{\aa} \coloneqq\begin{pmatrix}
	\begin{matrix}
	C_1& C_2&\ldots &C_{n+1} & C_{n+2}\\
	C_2& C_3&\ldots &C_{n+2} & 0\\
	\end{matrix}
 & \rvline & \begin{matrix}
	0 & 0 &\ldots &0\\
	C_{a_1}& C_{a_2}&\ldots &C_{a_{n-2}} \\
	\end{matrix}
    \end{pmatrix}.$$
    
Thus, $\calL_\aa$ is a $2n \times 2n$ square matrix divided into two blocks of size $n\times 2n$ each. We denote the top block by $\calL^{{\rm T,\aa}}$, and the bottom block by $\calL^{{\rm B,\aa}}$. When $\aa$ is clear from the context, we simply write $\calL^{{\rm T}}$ and $\calL^{{\rm B}}$.

    \begin{Remark}\label{rmk:detLisZero}
    We observe that $\det(\calL_{\aa})=0.$ In fact, after performing the row operations
    $$-R_k+R_{n+k-1}\rightarrow R_{n+k-1} \ \  \text{for} \ \ k=2,\ldots, n$$
    the bottom block becomes 
$$\calL^{{\rm B}}=\begin{pmatrix}
\begin{matrix}
    0&0 & 0& \cdots  & 0\\
    \vdots&\vdots& \vdots&\iddots  & \vdots\\
    0&0&0 &\cdots & 0\\
    x_{n+1} &x_{n+2}&0 &\cdots  & 0
\end{matrix}
& \rvline & \begin{matrix}
	C_{a_1}& C_{a_2}&\ldots &C_{a_{n-2}} 
	\end{matrix}
    \end{pmatrix}.$$
    Consequently, all the $n$-minors of $\calL^{{\rm B}}$ vanish, and it follows that the full determinant is zero.
\end{Remark}

    We now use generalized Laplace expansion by complementary minors to obtain a relation on the $n$-minors.  
    The determinant of $\calL_\aa$ can be expanded along the last $n$ rows as
    \begin{equation}\label{eq:LapMinors}
        \mathrm{det}(\calL_\aa)=\sum_{\Lambda\in \binom{[2n]}{n}} (-1)^{\sum_{\lambda\in \Lambda} \lambda+\frac{n(3n+1)}{2}} L^{{\rm B}}_{\Lambda} L^{{\rm T}}_{\Lambda^c} , 
    \end{equation}
    where $L^{{\rm B}}_{\Lambda}$ denotes the $n$-minor in the bottom block determined by the columns indexed by $\Lambda$, and ${L^{{\rm T}}_{\Lambda^c}}$ its complementary $n$-minor from the top block determined by the columns indexed by $\Lambda^c =[n+2] \setminus \Lambda$.
    The sign is determined by the sum of all indices corresponding to the rows and columns used to obtain the minor $L^{{\rm B}}_{\Lambda}$. The rows are always $n+1, \ldots, 2n $ and the columns are determined by $\Lambda$. \par 
        By the structure of the matrix, any nonzero $n$-minor in the lower block necessarily involves the columns $C_{\aa}=\{C_{a_1}, \ldots, C_{a_{n-2}}\}$, together with two additional ones, say $C_i$ and $C_j$. Indeed, if an $n$-minor in the lower block does not contain all the columns $C_{\aa}$, then its complementary minor in the upper block vanishes, and therefore, it does not contribute a new term to the expansion. Thus, we always have $$L_{\Lambda}^{{\rm B}}=\det(C_i,C_j,C_{a_1}, \ldots, C_{a_{n-2}})\quad \text{and} \quad L_{\Lambda^c}^{{\rm T}}=\det(C_1  \ldots  \widehat{C_{i-1}}  \ldots  \widehat{C_{j-1}}  \ldots  C_{n+2}),$$ where $\ \widehat{}\ $ denotes omission. Moreover, the sign of each of these summands depends only on $i$ and $j$ and is given by $$(-1)^{\frac{n(3n+1)}{2}+i+j-2+\frac{3(n-2)(n+1)}{2}}=(-1)^{i+j+1}.$$

This motivates the following definition.

\begin{Definition} \label{def:Lap}
Let $\mathbf{a}=\{a_1, a_2, \ldots, a_{n-2}\} \in \binom{[n+2]}{n-2}$ be a subset of $[n+2]$ of cardinality $n-2$. Set $B:=\{2,\ldots, n+2\} \setminus \mathbf{a}$. We define the \emph{Laplace relation} associated to $\aa$ to be the polynomial 
       \begin{equation}\label{eq:lap}
     \Lap_{\mathbf{a}}:=\sum\limits_{\{i,j\} \subset B} \sgn (i,j) \ T_{\{i,j\}\cup \mathbf{a}} T_{[n+2] \setminus \{i-1,j-1\}} \ \in \ \mathbb{K}[T_\mathbf{s} \mid \mathbf{s}=\{s_1,\ldots, s_n\}\subset [n+2]],
        \end{equation}
where $\sgn (i,j)=(-1)^{i+j+1 +\eta_i+\eta_j} $ with $\eta_k:= | \{\,a_s\in\aa \mid a_s<k\,\}| $. Equivalently, $\eta_k$ is the number of entries of $\aa$ that lie strictly to the left of $k$.
        \end{Definition}

The correction term $\eta_i+\eta_j$ in the sign of the Laplace relation compensates for the sign change produced when the columns indexed by $i$ and $j$ are reordered among the indices in $\aa$. Since the variables $T_{\mathbf{i}}$ are defined up to a permutation of their index set, we must include this term to ensure that each product of minors is written with the correct sign.
\medskip

As a direct consequence, we have that Laplace relations lie in $\calK[n-1]$.

   \begin{Lemma}\label{lem:LAP-vanish}
    $\Lap_{\aa}\in \calK[n-1]$ for every $\aa  \in \binom{[n+2]}{n-2}$, that is, $\psi(\Lap_\mathbf{a})=0$.
\end{Lemma}

\begin{proof}
    By Definition \ref{def:Lap}, we have $\psi(\Lap_\mathbf{a})=\det \calL_\aa$.
 The claim then follows from Remark \ref{rmk:detLisZero}.
\end{proof}
\begin{Example}\label{ex:lap}
      Let $n=4$ and fix $\aa=\{5,6\} \in \binom{[6]}{2}$. The $8 \times 8 $ matrix corresponding to this choice of $\aa$ is 

    $$\calL_\aa= \begin{pmatrix}
	\begin{matrix}
	C_1& C_2&C_3&C_4 &C_{5} & C_{6}\\
	C_2& C_3& C_4 &C_{5}&C_{6} & 0\\
	\end{matrix}
 & \rvline & \begin{matrix}
	0 & 0\\
	C_{5}& C_{6} \\
	\end{matrix} \\
\end{pmatrix}.$$
The generalized Laplace expansion along the last $4$ rows yields:

$$-[2456]^2+[2356][3456]+[1456][3456],$$

and the corresponding Laplace relation is given by

$$\Lap_{5,6}=-T_{2456}^2+T_{2356}T_{3456}+T_{1456}T_{3456}.$$

 \end{Example}
 
We now prove that the Laplace relations, together with the Pl\"ucker relations, generate  $\calK[n-1]$. In doing so, we will show that  $\calK[n-1]$ is isomorphic to the special fiber of the ideal of $2$-minors of  

$$ \LL= \begin{pmatrix}
    x_1 & x_2 &  x_3 & \cdots & x_{n+2}\\
    x_2 & x_3 &  \cdots & x_{n+2}& 0
\end{pmatrix}.$$

The latter special fiber arises as a particular case of the special fibers of $2$-determinantal ideals studied in greater generality in \cite{RS}. To this aim, inspired by the classical duality property of the Grassmannian $Gr(n,n+2) \isom Gr(2,n+2)$, we construct a $\KK$-algebra isomorphism 
\begin{align} \label{eq:Kisom}
        \varphi: \mathbb{K}[T_\bfi \mid \bfi=\{i_1,\ldots, i_n\}\subset [n+2]]&\longrightarrow \KK[T_{k_1,k_2} \mid \{k_1,k_2\}\subset [n+2]]
    \end{align}

that preserves the poset structure of the underlying minors. Recall that the set $\Pi_{n,n+2}$ of maximal minors of a generic $n\times (n+2)$ matrix is partially ordered by the relation:
$$[i_1\ldots i_n] \preceq [j_1 \ldots j_n] \Leftrightarrow i_k \leq j_k \ \text{ for all} \ k=1, \ldots, n.$$
The construction of  $\varphi$ is as follows. 
Given a subset of indices  $\kk \in [N]$, denote its complement by $\kk^c:= [N]\setminus \kk$, and write $c(\kk)=\kk^c$ for the complement map.
 
Let $\sigma$ be the longest element in the symmetric group $\mathfrak{S}_N$ with respect to the generating set consisting of the adjacent transpositions $(i\;i+1)$; in one-line notation $\sigma=(N\ N-1 \ldots1)$.   
We define 
$$\varphi(T_\bfi):= T_{\sigma(\bfi^c)},$$ where $N=n+2$. It is easy to verify that $\sigma\circ c=c\circ \sigma$, i.e., $\sigma(\kk^c)=\sigma(\kk)^c$. Moreover, both $c$ and $\sigma$ are involutions, meaning that $\sigma\circ \sigma =\id $ and $c \circ c=\id $. Thus, $(\sigma \circ c)^{-1}=c \circ \sigma $ and we have $\varphi^{-1} (T_{\{i,j\}})= T_{\sigma(\{i,j\}^c)} $.\par

\begin{Remark}\label{rmk:isoposet}
The composition $\sigma  \circ  c$ naturally induces an isomorphism on the corresponding posets of maximal minors $$\sigma  \circ  c: \Pi_{n,n+2}\longrightarrow \Pi_{2,n+2}.$$ 
Indeed, both $c$ and $\sigma$ are bijective and order-reversing, so their composition is bijective and order-preserving. In  other words,  $\varphi$ identifies the variables corresponding to the same element in the respective posets.

\end{Remark}

We illustrate this isomorphism with an example.
\begin{Example}\label{ex:diagram}
Consider the two posets of maximal minors $\Pi_{4,6}$ and $\Pi_{2,6}$ in Figure ~\ref{fig:poset}.

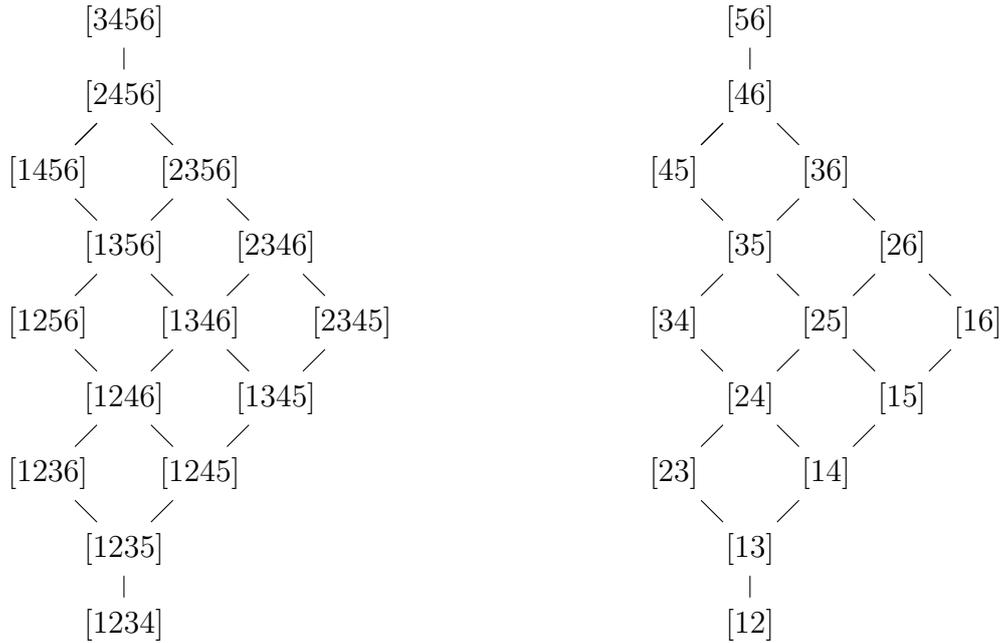
\begin{figure}[h]
 \centering
    \begin{subfigure}{0.45\textwidth}
     \centering
        \begin{tikzpicture}[scale=1.0,auto=center]
\node (n1) at (2,1) {[1234]};
\node (n2) at (2,2)  {[1235]};
\node (n3) at (1,3)  {[1236]};
\node (n4) at (3,3) {[1245]};
\node (n5) at (2,4)  {[1246]};
\node (n6) at (4,4)  {[1345]};
\node (n7) at (1,5)  {[1256]};
\node (n8) at (3,5)  {[1346]};
\node (n9) at (5,5)  {[2345]};
\node (n10) at (2,6)  {[1356]};
\node (n11) at (4,6)  {[2346]};
\node (n12) at (1,7)  {[1456]};
\node (n13) at (3,7)  {[2356]};
\node (n14) at (2,8)   {[2456]};
\node (n15) at (2,9)   {[3456]};
\foreach \from/\to in {n1/n2,n2/n3,n2/n4,n3/n5,n4/n5,n4/n6,n5/n7,n5/n8,n6/n8,n6/n9,n7/n10,n8/n10,n8/n11,n9/n11,n10/n12,n10/n13,n11/n13,n12/n14,n12/n14,n13/n14,n14/n15}
\draw (\from) -- (\to);
\end{tikzpicture}
    \end{subfigure}
    \begin{subfigure}{0.45\textwidth}
     \centering
\begin{tikzpicture}[scale=1.0,auto=center]
\node (n1) at (2,1) {[12]};
\node (n2) at (2,2)  {[13]};
\node (n3) at (1,3)  {[23]};
\node (n4) at (3,3) {[14]};
\node (n5) at (2,4)  {[24]};
\node (n6) at (4,4)  {[15]};
\node (n7) at (1,5)  {[34]};
\node (n8) at (3,5)  {[25]};
\node (n9) at (5,5)  {[16]};
\node (n10) at (2,6)  {[35]};
\node (n11) at (4,6)  {[26]};
\node (n12) at (1,7)  {[45]};
\node (n13) at (3,7)  {[36]};
\node (n14) at (2,8)   {[46]};
\node (n15) at (2,9)   {[56]};
\foreach \from/\to in {n1/n2,n2/n3,n2/n4,n3/n5,n4/n5,n4/n6,n5/n7,n5/n8,n6/n8,n6/n9,n7/n10,n8/n10,n8/n11,n9/n11,n10/n12,n10/n13,n11/n13,n12/n14,n12/n14,n13/n14,n14/n15}
\draw (\from) -- (\to);
\end{tikzpicture}
    \end{subfigure}
\caption{Isomorphic posets of maximal minors $\Pi_{4,6}$ and $\Pi_{2,6}$}
\label{fig:poset}
\end{figure}
The isomorphism $\varphi$ identifies variables representing minors that correspond to the same element in the respective posets. For example, $\varphi(T_{1346})=T_{25}$ and $\varphi(T_{1245})=T_{14}$.
\end{Example}

In Lemma \ref{lem:isom} we show that under $\varphi$ the Laplace relations \eqref{eq:lap} and the Plucker relations of $\Gr(n, n+2)$ are mapped bijectively onto the Laplace relations in \cite{RS} and the Plucker relations of $\Gr(2, n+2)$, respectively; these images together generate the defining ideal of $\calF(I_2(\LL))$. \par 
The following observation on the structure of the Laplacian relations \eqref{eq:lap} will be useful in proving this result.

\begin{Remark} \label{rmk:cardB}
    There are exactly $\binom{n+2}{n-2}=\binom{n+2}{4}$ Laplace relations. The number of summands in each relation depends on whether  $1 \in \mathbf{a}$ or not. More precisely: \par 
    \begin{itemize}
        \item If $1 \in \mathbf{a}$, then $|B|=4$. Hence, $\text{LAP}_{\mathbf{a}}$ has $\binom{4}{2}=6$ summands. There are $\binom{n+1}{n-3}$ such equations.
        \item If $1 \notin \mathbf{a}$, then $|B|=3$. Hence, $\text{LAP}_{\mathbf{a}}$ has $\binom{3}{2}=3$ summands. There are $\binom{n+1}{n-2}$ such equations.
    \end{itemize}
    
\end{Remark}

\begin{Lemma}\label{lem:isom}
    The ideal generated by the Pl\"ucker \eqref{eq:general-plucker} and Laplace relations \eqref{eq:lap} is isomorphic via $\varphi$ to the defining ideal of $\calF(I_2(\LL))$.
\end{Lemma}
\begin{proof}

To prove the lemma, we show that for each Laplace relation $\Lap_\aa$, its image $\varphi(\Lap_\aa)$ coincides with exactly one of the Laplace relations described in \cite{RS}. We consider two cases.\par 
First assume that $1 \notin \mathbf{a}$. By Remark \ref{rmk:cardB} we have $B=\{\alpha, \beta, \gamma\}$ with
$1<\alpha< \beta<\gamma$, and 
\[ \text{LAP}_{\aa}= \sgn(\alpha,\beta)T_{\aa\cup \{\alpha,\beta\}}T_{\{\alpha-1, \beta-1\}^c}+\sgn(\alpha,\gamma)T_{\aa\cup \{\alpha,\gamma\}}T_{\{\alpha-1, \gamma-1\}^c}+\sgn(\beta,\gamma)T_{\aa\cup \{\beta,\gamma\}}T_{\{ \beta-1,\gamma-1\}^c}. \]

Notice that $(\aa\cup \{\alpha,\beta\})^c=\{1,\gamma\}$, and similarly for the other terms. Applying $\varphi$ yields: 

\begin{align*}
    \varphi(\text{LAP}_{\aa})=& \sgn(\alpha,\beta)  T_{\sigma(\{1, \gamma\})}\ T_{\sigma(\{\alpha-1,\beta-1\})}+\sgn(\alpha,\gamma) T_{\sigma(\{1,\beta\})}\ T_{\sigma(\{\alpha-1, \gamma-1\})}+\\
    &+\sgn(\beta,\gamma) T_{\sigma(\{1,\alpha\})}T_{ \sigma(\{\beta-1,\gamma-1\})} \\
    =&\sgn(\alpha,\beta) T_{n+3-\gamma, n+2}\ T_{n+4-\beta, n+4-\alpha}+\sgn(\alpha,\gamma) T_{n+3-\beta, \ n+2}\ T_{n+4-\gamma, n+4-\alpha}+\\
    \quad & +\sgn(\beta,\gamma) T_{n+3-\alpha, n+2}T_{ n+4-\beta,n+4-\gamma}. 
\end{align*}
These relations coincide with the relations (5.6) in \cite{RS} obtained from  $$\det \begin{pmatrix}
    C_{n+3-\gamma } & C_{n+3-\beta} & C_{n+3-\alpha}  & C_{n+2}\\
C_{n+4-\gamma } & C_{n+4-\beta}  & C_{n+4-\alpha} & 0
\end{pmatrix}=0.$$
In fact, since $\aa \cup \{\alpha, \beta, \gamma\}=\{2,\ldots, n+2\}$, it can be easily checked that
$$\eta_\alpha=\alpha-2, \ \  \eta_{\beta}=\beta-3, \ \  \text{and} \ \ \eta_{\gamma}=\gamma-4.$$
Therefore,
$$ \sgn(\alpha,\beta)=\sgn(\beta,\gamma)=- \sgn(\alpha,\gamma)=1.$$
Assume now that $1 \in \aa$. A similar computation shows that the relations $\Lap_\aa$ correspond to the relations (5.5) in \cite{RS}. 
The argument is analogous to the previous case and is omitted for brevity. \par
It is straightforward to check that $\varphi$ induces a bijection between the Pl\"ucker relations of $\Gr(n,n+2)$ and those of $\Gr(2,n+2)$. This completes the proof.
\end{proof}

Using Lemma \ref{lem:LAP-vanish} and Lemma \ref{lem:isom}, we get the defining equations of the special fiber $\calF[n-1]$.
\begin{Theorem}\label{thm:eqSF}
    The Pl\"ucker relations \eqref{eq:general-plucker} and the Laplace relations \eqref{eq:lap} generate the ideal $\calK[n-1]$ of the special fiber ring $\calF[n-1]$. In particular, $\calF[n-1]\isom \calF(I_2(\LL))$.
\end{Theorem}

\begin{proof}
Let $\calK'$ be the ideal generated by the Pl\"ucker and Laplace relations. By Remark \ref{rmk:chainF} and Lemma $\ref{lem:LAP-vanish}$, $\calK' \subseteq \calK[n-1]$. Moreover, $\calK'$ is isomorphic to the defining ideal of the special fiber ring $\calF(I_2(\LL))$ by Lemma \ref{lem:isom}. In particular, $\calK'$ is prime and by \cite[Proposition A.1] {RS} it has codimension
$$\Ht(\calK')={n+2 \choose 2}-(n+2).$$ By Remark \ref{rmk:maxanspread}, $\calF[n-1]$ has maximal analytic spread so $\Ht (\calK[n-1])=\Ht (\calK')$. Hence, we obtain that $\calK'=\calK[n-1].$ Notice that both algebras are defined over the same polynomial ring and their defining equations are isomorphic via $\varphi$, so $$ \calF[n-1]\isom \calF(I_2(\LL)).\eqno\qedhere $$ 
\end{proof}

We illustrate the bijective correspondence between the Laplace relations \eqref{eq:lap} and the ones in \cite{RS} with an example.

\begin{Example}
Let $n=4$ and let $\Lap_{5,6}$ be the Laplace relation in Example \ref{ex:lap}, that is
$$\Lap_{5,6}=-T_{2456}^2+T_{2356}T_{3456}+T_{1456}T_{3456}.$$

Applying $\varphi$ (see Figure \ref{fig:poset}), we get

$$\varphi(\Lap_{5,6})=-T_{46}^2+T_{36}T_{56}+T_{45}T_{56}.$$

This corresponds to the Laplace equation in \cite{RS} (in their notation, $\Lap_{3,4,5,6}$) obtained from the vanishing of the following minor

$$\det \begin{pmatrix}
    C_{3 } & C_{4} & C_{5}  & C_{6}\\
C_{4} & C_{5}  & C_{6} & 0
\end{pmatrix}.$$

\end{Example}

As a direct consequence of Theorem \ref{thm:eqSF} and \cite[Proposition A.1]{RS}, we get the following:
\begin{Corollary}\label{mult}
     $\calF[n-1]$ is Cohen-Macaulay and it has multiplicity $2^{n+1}-n-2$.
\end{Corollary}

Moreover, we can define a term order $\preceq$ on $\mathbb{K}[T_\mathbf{s} \mid \mathbf{s}=\{s_1,\ldots, s_n\}\subset [n+2]]$ such that $T_{\mathbf{s}}\preceq T_{\mathbf{s'}}$ if $\varphi(T_{\mathbf{s}})\preceq_{RS} \varphi(T_{\mathbf{s'}})$ where $\preceq_{RS}$ denotes the term order defined in \cite{RS}. By \cite[Theorem 6.1]{RS}, it follows that $\calF[n-1]$ is Koszul.

\begin{Corollary}\label{cor:SFKoszul}
    The Pl\"ucker and Laplace relations form a Gr\"obner basis of $\calK[n-1]$ with respect to $\preceq$. In particular, the special fiber $\calF[n-1]$ is Koszul.
\end{Corollary}

We conclude by observing that the isomorphism in Theorem \ref{thm:eqSF} allows us to compute the reduction number of $I_n(H[n-1])$ and the $a$-invariant of its special fiber. 

\begin{Corollary} \label{cor:alginvariant r=n-1}The reduction number of $I_n(H[n-1])$, which coincides with the regularity of its special fiber, is 
    $${\rm{r}}(I_n(H[n-1]))= \reg (\calF[n-1])= \floor{\dfrac{n+2}{2}}.$$
As a consequence, the $a$-invariant of $I_n(H[n-1])$ is given by
$${\rm a}(\calF[n-1])=\floor{\dfrac{n+2}{2}}- (n+2).$$
\end{Corollary}
\begin{proof} By Theorem \ref{thm:eqSF} and Corollary \ref{mult}, $\calF[n-1]\cong \calF(I_2(\LL))$ and both rings are Cohen-Macaulay. This implies that ${\rm{r}}(I_n(H[n-1]))={\rm{r}}(I_2(\LL))$. As noted at the end of \cite[Section 8]{RS}, we can apply \cite[Corollary 4.8]{CHV} to compute ${\rm{r}}(I_2(\LL))$. Using the terminology of \cite{RS}, the ideal $I_2(\LL)$ is an ideal of KW type $(\emptyset,n+2)$, meaning that its codimension is $n+1$ and the corresponding matrix has no scroll blocks (i.e., $d=0$), and exactly one Jordan block. Applying \cite[Corollary 4.8]{CHV} with $d=0$ and $c=n+2$, we obtain:
$${\rm{r}}(I_n(H[n-1]))={\rm{r}}(I_2(\LL))= \floor{\dfrac{n+2}{2}}.$$

Since $\calF[n-1]$ is Cohen-Macaulay, we have ${\rm{r}}(I_n(H[n-1]))={\reg}(\calF[n-1])$. Moreover,
$${\rm a}(\calF[n-1])= {\reg(\calF[n-1])}- \ell(I_n(H[n-1]))=\floor{\dfrac{n+2}{2}}- (n+2).\eqno\qedhere$$
\end{proof}

\subsection{Application to birational maps}\label{sec:geometry}

We now provide a brief discussion on the geometric counterpart of the previous results. Motivated by previous works on homaloidal hypersurfaces and determinants (see \cite{CRS, MS}), we prove that the projective maps associated with certain reductions of $I_n(\calH[n-1])$ are birational maps. We call these reductions maximal reductions (see Remark \ref{rmk:maxred} and Definition \ref{def:maxred}). Indeed, it is known that for generic Hankel matrices the polar maps associated with the determinant are never birational \cite{Jeff, MS}. Inspired by \cite[Lemma 30]{Huh} (see Equation \eqref{eq:deg-equality} and accompanying discussion), rather than working with the gradient ideal itself, we consider the coordinate sections of the gradient ideal of the generic Hankel matrix, which leads to birational maps in this setting.

\vspace{0.5cm}

\noindent
We now recall some general facts about mixed multiplicities. Let $I=(f_0,\ldots,f_s) \subset S$ be an ideal generated by homogeneous forms of the same degree in $t+1$ variables. The \emph{rational map associated to $I$} is the map on projective space defined by the generators of $I$; we denote it by 
\begin{equation}\label{eq:rationalmap}
 \Theta_I : \mathbb{P}^t \dashrightarrow \mathbb{P}^s.   
\end{equation}

It is well known that the mixed multiplicities of $\mathfrak{m}$ and $I$, denoted by  $e_i(\mathfrak{m} \mid I)$,  coincide with the projective degrees of $\Theta_I$, where $\mathfrak{m}$ is the maximal homogeneous ideal of $S$ (see \cite[Example~19.4]{Harris} and \cite[Remark~4]{Huh}). More precisely, the mixed multiplicity sequence $e_i(\mathfrak{m} \mid I)$ consists of nonnegative integers whose last (or, in some literature, first) nonzero term satisfies
$$
e_{\ell}(\mathfrak{m} \mid I) = \deg(\Theta_I) \cdot e(\mathcal{F}(I)).
$$
 
If  $\Ht (I)>0$ then  $\ell = \ell(I) - 1$ where $\ell(I)$ is the analytic spread of $I$, \cite{KV89}.  

The following result shows that mixed multiplicities are preserved under reductions.

\begin{Lemma}[\cite{Huh}, Lemma~30]
Let $I, J$ be ideals of $S$ such that $J$ is a reduction of $I$. Then 
$$
e_i(\mathfrak{m} \mid I) = e_i(\mathfrak{m} \mid J) \quad \text{for all } 0 \leq i \leq \dim S - 1.
$$
\end{Lemma}

In particular, the last nonzero entries of the two sequences coincide. Consequently, we obtain
\begin{equation} \label{eq:deg-equality}
\deg(\Theta_I) \cdot e(\mathcal{F}(I)) = \deg(\Theta_J) \cdot e(\mathcal{F}(J)). 
\end{equation}
This latter identity is the key point of this section. 
\begin{Remark}\label{rmk:maxred}
If $J$ is a \emph{minimal reduction} of $I$, then $\mathcal{F}(J)$ is isomorphic to a polynomial ring and $e(\calF (J))=1$; consequently, $\mathrm{deg}(\Theta_J)$ is the largest possible, and $\Theta_J$ might not be birational. Conversely, if $J$ is a reduction of $I$ such that $\Theta_J$ is birational onto its image, then $\deg(\Theta_J) = 1$, and $e(\calF (J))$ attains its maximal value.
\end{Remark}

This observation motivates the following definition: 

\begin{Definition}\label{def:maxred}
Let $S$ be a polynomial ring over an infinite field $\mathbb{K}$, and let $J\subset I \subset S$ be ideals generated by forms of the same degree, with $J$ a reduction of $I$.  
We say that $J$ is a \emph{maximal reduction} of $I$ if $\deg(\Theta_J) = 1$.
\end{Definition}

\begin{Remark} If $I$ admits a maximal reduction, then $\ell(I) = \dim S$. Indeed, a maximal reduction $J$ defines a birational map. In particular, this map is dominant; hence it has maximal analytic spread, i.e. $\ell(J)=\dim \, S$. Since reductions preserve analytic spread, the claim follows. 
\end{Remark}

Now, we recall an algebraic tool to detect birational maps. Roughly speaking if $I$ has maximal analytic spread and \emph{sufficient} linear syzygies then $\Theta_I$ is a birational map.

\begin{Theorem}\label{Bir}\cite[Theorem~ 3.2]{DHS12} Let $\Theta_I: \mathbb{P}^t \dashrightarrow \mathbb{P}^s$ be the rational map of $I$ as in \eqref{eq:rationalmap}. If $\ell(I)=t+1$ and $\rank \,M_1$ is at least $t$,  then $\Theta_I$ is birational onto its image. Here, $M_1$ denotes the sub-matrix of the presentation matrix of $I$ consisting of all linear syzygies.
    
\end{Theorem}

We now apply these results to $I_n(H[n-1])$.
Let $F=\det(\mathcal{H})$ and denote by $J = (\nabla F)$ its gradient ideal. By \cite[Theorem~4.6]{CMRS}, $J$ is a minimal reduction of $I_n(\mathcal{H})$. Let $J[r]$ denote the specialization of $J$ obtained by setting $ x_{2n-r+2} = \dots = x_{2n+1} =0$ where $0< r\leq n-1$. Then $J[r]$ is a reduction of $I_n({H}[r])$, but it is no longer minimal. Applying identity~\eqref{eq:deg-equality} to $I_n := I_n(\mathcal{H}[n-1])$ and $J[n-1]$ yields

\begin{equation}\label{eq:teta-mult}
    \deg(\Theta_{I_n}) \cdot e(\mathcal{F}[n-1]) = \deg(\Theta_{J[n-1]}) \cdot e(\mathcal{F}(J[n-1])).
\end{equation}

 We will show that $\Theta_{I_n}$ and $\Theta_{J[n-1]}$ are birational maps; hence both have degree $1$. As a consequence, we compute $e(\mathcal{F}(J[n-1]))$.

 \begin{Lemma}\label{bir-ideal} Let $I=I_n({H}[r])$, then for any $0\leq r\leq n-1$ the rational map $\Theta_I$ is birational. 
 \end{Lemma}
\begin{proof}
    
The ideal $I_n(H[r])$ is perfect of expected codimension, so it has a minimal linear resolution given by the Eagon-Northcott complex. The rank of the sub-matrix of linear syzygies (here all syzygies are linear) is ${n+2 \choose 2}-1$.  By Remark \ref{rmk:maxanspread},  $\ell(I_n(H[r]))=2n+1-r$. The claim follows from Theorem \ref{Bir}. 
\end{proof}

\begin{Proposition}\label{bir-gradient}  $\Theta_{J[n-1]}$ is a birational map.
\end{Proposition}

\begin{proof}
  Since $J[n-1]$ is a reduction of $I_n(H[n-1])$, it has maximal analytic spread, that is $\ell(J[n-1])=n+2$. In \cite[Lemma~4.2 (iii)]{CRS}, a sub-matrix of linear syzygies of $J[n-1]$ is computed. More precisely, they provide linear syzygies of the gradient ideal $\nabla f^{(n+1)}$, where $f^{(n+1)}:={\det }(\calH[n-1])$:

\begin{equation}\label{LinSyz}
\begin{pmatrix}
 nx_1&(2n)x_2 & 2(n-1)\,x_3 &2(n-2)x_4 &\dots &&2x_{n+1}\\
(n-1) x_2&(2n-1)x_3 &(2n-3)x_4& (2n-5)x_5 &\dots&& x_{n+2}\\
\vdots &\vdots & \vdots & \vdots & \iddots &&\vdots\\
2x_{n-1}&(n+2)x_{n} & (n)x_{n+1} &(n-2)x_{n+2}& \dots&& 0\\
x_n&(n+1 ) ~~~~x_{n+1}& (n-1)x_{n+2} &0 & \dots && 0\\
0&(n)x_{n+2} & 0 & 0 & \dots && 0\\
-x_{n+2}&0&0&0&\cdots &&0
\end{pmatrix}_{(n+2)\times (n+1)}
\end{equation}

\medskip

Explicitly, $\nabla f^{(n+1)}=(\overline{\frac{\partial F}{\partial x_{1}}}, \ldots ,\overline{\frac{\partial F}{\partial x_{n+2}} })$. Note that $J[n-1]=\nabla f^{(n+1)}+ (\overline{\frac{\partial F}{\partial x_{n+3}}}, \ldots ,\overline{\frac{\partial F}{\partial x_{2n+1}} })$, so the linear syzygies of $\nabla f^{(n+1)}$ form a sub-matrix of the matrix of linear syzygies of $J[n-1]$.
  We claim that the rank of the matrix (\ref{LinSyz}) is exactly $n+1.$
Deleting the first row yields an upper triangular matrix with non-zero diagonal entries, provided that the characteristic of the base field is  bigger than $n$ or is zero. Hence, the rank of this matrix is $n+1$ and 
the claim follows from Theorem~\ref{Bir}. 
\end{proof}

\begin{Corollary} The ideal $J[n-1]$ is a maximal reduction of $I_n(H[n-1])$, and its multiplicity is given by $e(\calF (J[n-1]))=2^{n+1}-n-2.$ 
    
\end{Corollary}

\begin{proof}
Since $J[n-1]$ is a reduction of $I_n(H[n-1])$ and $\deg(\Theta_{J[n-1]})=1$, then $J[n-1]$ is a maximal reduction. By identity~\eqref{eq:teta-mult} and Lemma~\ref{bir-ideal},  we have $e(\mathcal{F}(J[n-1]))=e(\mathcal{F}[n-1])$. Finally, Corollary~\ref{mult} yields $e(\mathcal{F}(J[n-1]))=2^{n+1}-n-2,$ as claimed.
\end{proof}

\section{Degeneration \texorpdfstring{$r=1$}{r=1}}\label{sec:r=1}

We now turn our attention to the other extremal degeneration, namely when $r=1$ and 

\begin{align*}
H[1]=\begin{pmatrix}
x_1&x_2&\ldots & x_n& x_{n+1}& x_{n+2}\\
\vdots &\vdots  &\iddots &\vdots &\vdots & \vdots \\
x_{n-1}&x_{n}&\ldots & x_{2n-2} &x_{2n-1} &x_{2n}\\
x_{n}&x_{n+1}&\ldots &x_{2n-1} &x_{2n} &0\\
\end{pmatrix}.
\end{align*}

We again begin by determining the defining equations of the special fiber $\calF[1]$; equivalently, a set of generators for the ideal $\calK[1]$. In this case, we will see that a single additional relation, beyond the Pl\"ucker ones, suffices to generate $\calK[1]$. However, this generator has degree $n$. This result provides a negative answer to one of the questions posed in \cite[Question 2.9]{CMRS}.\par
We start by proving that $\calF[1]$ is Gorenstein and that $\calK[1]$ is generated by exactly one additional relation besides the Pl\"ucker ones.

\begin{Proposition} \label{lem:fiberGor}
    The special fiber ring $\calF[1]$ is Gorenstein and $\calK[1]$ is generated by the Pl\"ucker relations and exactly one regular element over $\calF[0]$.
\end{Proposition}
\begin{proof}
 As noted in Remark \ref{rmk:chainF}, the Pl\"ucker ideal is contained in $\calK[1]$. Consider the surjection
\begin{equation}\label{eq:pi1}
    \pi_1:  \calF[0] \twoheadrightarrow \calF[1].
\end{equation}
One has $\faktor{\calF[0]}{\ker(\pi_1)}\isom \calF[1]$. By Remark \ref{rmk:maxanspread} $\calF[1]$ has maximal analytic spread, so we have
$$\dim \faktor{\calF[0]}{\ker(\pi_1)}=\mathrm{dim}(\calF[1])=2n.$$
Moreover, it is known that $\calF[0]\isom\Gr(n,n+2)$ is a Cohen-Macaulay UFD, hence Gorenstein (see \cite[Corollary 3.2]{H}). 
This implies that 
$$\Ht({\ker(\pi_1)})=\mathrm{dim}(\calF[0])-\mathrm{dim}(\calF[1])=1.$$
Hence $\mathrm{ker}(\pi_1)$ is a prime ideal of height one in a UFD, so it is a principal ideal generated by a regular element $f$ over $\calF[0]$.  In particular, since $\calF[0]$ is Gorenstein, $\calF[1]$ is Gorenstein.
\end{proof}

We will show in Theorem \ref{thm:eqSFr=1} that the extra relation $f$ in Proposition \ref{lem:fiberGor} is a polynomial of degree $n$, arising from the generalized Laplace expansion of the following  $n^2 \times n^2$ matrix.

\begin{equation}\label{eq:polynomial f}
    \calL[1]\coloneqq\begin{pmatrix}
    C_1&C_2&\ldots  & C_{n+1}& C_{n+2} &0&0&0& 0 &  \ldots &0& \ldots & 0 \\
    C_2&C_3& \ldots &C_{n+2} & 0 & C_1&C_2 &\ldots&C_{n+1} & \ldots&0&\ldots& 0\\
         0& 0&0 &0 &0&C_2&C_3&\ldots&C_{n+2}&  \ldots&0&\ldots&0\\
     
         \vdots&&&&\vdots&&&& \vdots&  \ddots&&\vdots \\
         0&0&0&0&0&0&0&0&0&\ldots&  C_1&\ldots&C_{n+1}\\
         0&0&0&0&0&0&0&0&0  &  \ldots & C_2&\ldots&C_{n+2}
         
\end{pmatrix}.
\end{equation}

This matrix is organized into $n-1$ vertical blocks, each containing a non-zero sub-block that partially overlaps with the next, forming a staircase-like pattern. The first non-zero block is given by
$$\begin{pmatrix}
    C_1 &\ldots &C_{n+1}&C_{n+2}\\
    C_2 &\ldots &C_{n+2}&0\\
\end{pmatrix},$$ 
while each of the remaining $n-2$ non-zero blocks are given by
$$\begin{pmatrix}
    C_1 \ldots C_{n+1}\\
    C_2 \ldots C_{n+2}\\
\end{pmatrix}.$$ 

\begin{Lemma} The determinant of $\calL[1]$ is zero.
\end{Lemma}
\begin{proof}
If we replace row $n^2-n+1$ with the linear combination
$$
\sum_{j=1}^{n} (-1)^{j+1} R_{\,n^2-jn+j} \longrightarrow R_{\,n^2-n+1},
$$
this row becomes zero. Consequently, $\det \mathcal{L}[1] = 0.$
\end{proof}

By iteratively applying the generalized Laplace expansion by complementary minors, we obtain a relation among the $n$-minors. Specifically, the determinant of $\mathcal{L}[1]$ can be expanded starting from the $n$-minors of the last $n$ rows and proceeding inductively upward through the successive blocks of $n$ rows. In this way we obtain a relation of the form:

\begin{equation}\label{eq:determinant of L1}
    \det \mathcal{L}[1] =
\sum_{\substack{(A_1,\dots,A_n) \\ A_k \subseteq \{1,\dots,n^2\},\, |A_k|=n \\ A_i \cap A_j = \emptyset}}
\operatorname{sgn}(A_1,\dots,A_n)\;\prod_{k=1}^{n} \det \big(\calL^{(k)}_{A_k}\big),
\end{equation}

where $\calL^{(k)}_{A_k}$ is the $n\times n$ submatrix of the $k$-th row block with columns indexed by $A_k$. Note that many of the terms in this sum are actually zero, depending on the choice of the columns. 
\medskip

We denote by $f_{\Lap}$ the resulting degree-$n$ polynomial in $\mathcal{P}_{\mathcal{F}}$ corresponding to this relation. 
\medskip

The next example illustrates the construction of the polynomial $f_{\Lap}$.

\begin{Example}\label{ex:f-relation}
    Consider $n=3$, the above relation $f_{\Lap}$ is a cubic obtained by the vanishing of the following $9$-minor 
    $$\calL[1]=\begin{pmatrix}
    C_1&C_2&C_3& C_4 &C_5& 0& 0 & 0 & 0 \\
    C_2&C_3&C_4& C_5 &0& C_1&C_2 &C_3&C_4\\
         0& 0&0  &0 &0&C_2&C_3&C_4&C_5
         \end{pmatrix}.$$

Expanding along the $3$-minors of the last block of $3$ rows, we obtain:
\begin{align*}
    \det \calL[1]=&[345]\det \begin{pmatrix}
    C_1&C_2&C_3& C_4 &C_5& 0\\
    C_2&C_3&C_4& C_5 &0& C_1\\
\end{pmatrix}- [245]\det \begin{pmatrix}
    C_1&C_2&C_3& C_4 &C_5& 0\\
    C_2&C_3&C_4& C_5 &0& C_2\\
\end{pmatrix} +\\
&[235]\det \begin{pmatrix}
    C_1&C_2&C_3& C_4 &C_5& 0\\
    C_2&C_3&C_4& C_5 &0& C_3\\
\end{pmatrix}-[234]\det \begin{pmatrix}
    C_1&C_2&C_3& C_4 &C_5& 0\\
    C_2&C_3&C_4& C_5 &0& C_4\\
\end{pmatrix}.
\end{align*}
We can then expand each of these four $6$-minors further. These expansions have already been computed in Equation \eqref{eq:LapMinors}. Using the Laplace polynomials in Definition $\ref{def:Lap}$, the corresponding algebraic relation can be written as
$$f_{\Lap}=T_{345} \Lap_1-T_{245}\Lap_2+T_{235}\Lap_3-T_{234}\Lap_4.$$
Note that $f_{\Lap} \in \calK[n-1]$, in accordance with Remark \ref{rmk:chainF}. However, it is worth emphasizing that in this case the Laplace polynomials are \textit{not} quadratic relations among the $3$-minors, since the $6$-minors appearing above are all non-zero. The explicit expression of $f_{\Lap}$ is  

\begin{align*}
    f_{\Lap}=&T_{345} (+T_{123}T_{345}-T_{124}T_{245}+T_{125}T_{235}+T_{134}T_{145}-T_{135}^2+T_{145}T_{125})\\
    &-T_{245}(T_{234}T_{145}-T_{135}T_{235}+T_{245}T_{125})+T_{235}(-T_{235}^2+T_{234}T_{245}+T_{345}T_{125})\\
    &-T_{234}(T_{234}T_{345}+T_{135}T_{345}-T_{245}T_{235})\\
    =& T_{123}T_{345}^2 -T_{124}T_{245}T_{345} +T_{125}T_{145}T_{345}
   +2T_{125}T_{235}T_{345} -T_{125}T_{245}^2 \\
&\quad +T_{134}T_{145}T_{345} -T_{135}^2T_{345}
   -T_{135}T_{234}T_{345} +T_{135}T_{235}T_{245}
   -T_{145}T_{234}T_{245} \\
&\quad -T_{234}^2T_{345} +2T_{234}T_{235}T_{245}
   -T_{235}^3.
\end{align*}
\end{Example}

By definition, $\psi(f_{\Lap})=\det \calL[1]=0$. Thus, $f_{\Lap} \in \calK[1]$. Our goal is to show that this polynomial, together with the Pl\"ucker relations ~\eqref{eq:general-plucker}, 
minimally generates $\calK[1]$ (see Theorem \ref{thm:eqSFr=1}).
\medskip

We begin with some observations on the structure of the polynomial $f_{\Lap}$ that will be useful to prove  Theorem \ref{thm:eqSFr=1}.

\begin{Proposition}\label{prop:structure of f} The polynomial $f_{\Lap} \in \calP_\calF$ can be written as
$$f_{\Lap}=-T_{2,3,\ldots,n+1}^{\,n-1}\,T_{3,4,\ldots,n+2}-
T_{2,3,\ldots,n,n+2}^{\,n}+ f',$$
where every monomial in $f'$ involves at least one variable $T_\bfi$ different from the first three above.
 \end{Proposition}

 \begin{proof}
     First, we observe that the two monomials $T_{2,3,\ldots,n+1}^{\,n-1}\,T_{3,4,\ldots,n+2}$ and $
T_{2,3,\ldots,n,n+2}^{\,n}$ appear exactly once in the Laplace expansion, and thus cannot be canceled by any other term.\par 
   The monomial $T_{2,3,\ldots,n,n+2}^{\,n}$ corresponds to choosing, in every row, the $n$-minor on columns
  $$\begin{pmatrix}
       C_2&C_3& \cdots & C_{n}&C_{n+2} 
   \end{pmatrix}.$$
At each step this choice is uniquely determined by the block structure of $\calL[1]$. 

Similarly, the monomial $T_{2,3,\ldots,n+1}^{\,n-1}\,T_{3,4,\ldots,n+2}$  corresponds to choosing once the minor 
$$\begin{pmatrix}
C_3,C_4,\ldots,C_{n+2} 
\end{pmatrix},$$
and $n-1$ times the minor
$$\begin{pmatrix}C_2,C_3,\ldots,C_{n+1}
\end{pmatrix}.$$  
By the structure of $\calL[1]$, $C_{n+2}$ must be always chosen from the first block of rows. Hence, all the other choices needed to form the monomial $T_{2,3,\ldots,n+1}^{\,n-1}\,T_{3,4,\ldots,n+2}$ are uniquely determined.

We note that the signs of these two monomials can be determined using the usual sign formula in the generalized Laplace expansion; however, since they play no role in our arguments, we omit their explicit computation.

   We now prove that any other monomial in the expansion contains at least one variable $T_{\bfi}$ different from the three distinguished ones. Suppose there is a term in the expansion that gives a monomial
   $$M:=T^{\alpha}_{2,3,\ldots,n+1} T^{\beta}_{3,4,\ldots,n+2}T^{\gamma}_{2,3,\ldots, n,n+2}.$$
   We will show that there are only two possibilities for $M$:
   $$T^{n-1}_{2,3,\ldots,n+1} T_{3,4,\ldots,n+2} \quad \text{or} \quad T^{n}_{2,3,\ldots, n,n+2}.$$
   As observed before, in each term of the expansion we always have a minor containing the column $C_{n+2}$ coming from the first block of rows. Therefore, the only possible minors corresponding to the three variables in $M$ are 
   $$\begin{pmatrix}
       C_3&C_4& \cdots &C_{n+2} 
   \end{pmatrix}\quad  \text{or} \quad  \begin{pmatrix}
       C_2&C_3& \cdots & C_{n}&C_{n+2} 
   \end{pmatrix}.$$ 
   Suppose that we have  $\left( C_2\ C_3\  \cdots \ C_{n}\ C_{n+2} \right)$ from the first row. This means that in the second block of rows we selected a minor containing both $C_2$ and $C_{n+2}$. The only such minor among those corresponding to  variables in $M$ is exactly the same, that is  $\left( C_2\ C_3\  \cdots \ C_{n}\ C_{n+2} \right)$. This pattern repeats in each subsequent row, yielding the unique term $T^{n}_{2,3,\ldots, n,n+2}$. 

Alternatively, suppose our term in the expansion contains $\begin{pmatrix}
       C_3&C_4& \cdots &C_{n+2} 
   \end{pmatrix}$ from the first block of rows. This means that in the second block of rows we selected a minor containing $C_2$ and $C_3$, but not $C_{n+2}$. The only choice consistent with $M$ is  $\begin{pmatrix}
       C_2&C_3& \cdots &C_{n+1} 
   \end{pmatrix}.$ Proceeding in this way through all subsequent rows, the selection is uniquely determined, resulting in the term 
   $$T^{n-1}_{2,3,\ldots,n+1} T_{3,4,\ldots,n+2}.\eqno\qedhere$$
 \end{proof}

\begin{Remark} \label{purePower}The polynomial $f_{\Lap}$ is not in the Pl\"ucker ideal. Indeed all monomials in the quadratic Plücker relations are square-free, whereas $f_{\Lap}$ contains a pure power of the variable $T_{2,3,\ldots,n,n+2}$.
    
\end{Remark}

\begin{Example}
If $n=4$ the polynomial $f_{\Lap}$ is a quartic arising from the Laplace expansion of the following $16 \times 16$ matrix:

$$\calL[1]=\begin{pmatrix}
    C_1&C_2&C_3& C_4 &C_5& C_6 &0&0&0& 0 & 0 & 0&0&0&0&0\\
    C_2&C_3&C_4& C_5 &C_6&0& C_1&C_2 &C_3&C_4&C_5&0&0&0&0&0\\
         0& 0&0 &0 &0&0&C_2&C_3&C_4&C_5&C_6&C_1&C_2&C_3&C_4&C_5\\
         0&0&0&0&0&0&0&0&0&0&0 &C_2 &C_3&C_4&C_5&C_6
\end{pmatrix}.$$ 
Expanding $\det(\mathcal{L}[1])$ along the last $4$ rows, we get
$$
\det(\mathcal{L}[1])
= 
[3456]\, \calM_{1}
- [2456]\, \calM_{2}
+ [2356]\, \calM_{3}
- [2346]\, \calM_{4}
+ [2345]\, \calM_{5},
$$

where $\calM_{i}$ denotes the complementary $12\times 12$ minor
$$\calM_{i}=\begin{pmatrix}
    C_1&C_2&C_3& C_4 &C_5& C_6 &0&0&0& 0 & 0 & 0\\
    C_2&C_3&C_4& C_5 &C_6&0& C_1&C_2 &C_3&C_4&C_5&0\\
         0& 0&0 &0 &0&0&C_2&C_3&C_4&C_5&C_6&C_i\\
\end{pmatrix}.$$ 

The complete Laplace expansion yields a polynomial $f_{\Lap} \in \calP_\calF$ whose explicit expression is too lengthy to write here. We therefore record only the terms that will be relevant in the proofs of the subsequent results:

$$f_{\Lap}= -T_{2346}^4-T_{2345}^3T_{3456}+ {\text {\rm{other terms}}}.$$

\end{Example}

We can now show that $f_{\Lap}$ is the sole extra equation, beyond the Pl\"ucker relations, needed to generate the defining ideal of $\mathcal{F}[1]$.

\begin{Theorem}\label{thm:eqSFr=1} The defining ideal of the special fiber of $I_n(H[1])$ is $\calK[1]=\mathscr{P}_{n,n+2}+(f_{\Lap})$.
    
\end{Theorem}

\begin{proof}
    Let $\pi_1$ be the surjection in \eqref{eq:pi1}. We show that $f_{\Lap}$ is irreducible modulo the ideal $\mathscr{P}_{n,n+2}$. Since $\mathcal{F}[0]$ is a UFD, the ideal generated by an irreducible element is prime, and thus $\mathrm{ker}\,\pi_1 = (f_{\Lap})$. This implies that $\calK[1]=\mathscr{P}_{n,n+2}+(f_{\Lap})$. In what follows we write  $\mathscr{P}$ for $\mathscr{P}_{n,n+2}$ to simplify the notation. 

 From Proposition~\ref{prop:structure of f} and Remark \ref{purePower} we know that  $f_{\Lap} \notin \mathscr{P}$ and can be written as 

$$f_{\Lap}=T_{2,3,\ldots,n+1}^{\,n-1}\,T_{3,4,\ldots,n+2}+
T_{2,3,\ldots,n,n+2}^{\,n}+ f'$$ 
 
where every monomial in $ f'$ involves at least one variable $T_\bfi$ not among 
$$T_{2,3,\ldots,n+1}, \ T_{3,4,\ldots,n+2}, \ T_{2,3,\ldots,n,n+2}.$$ 
To show that $f_{\Lap}$ is irreducible modulo $\mathscr{P}$, we will specialize all the variables different from the three above to zero. Thereby, we obtain an irreducible polynomial in a smaller ring, and we deduce that $f_{\Lap}$ is irreducible over $\mathscr{P}$. Note that the signs of the two (unique) monomials in the expansion will be irrelevant for the proof. For convenience, we therefore assume that both terms have positive sign. Let 
$$ \mathscr{T}:=
\left( T_{\bfi} \mid \bfi \notin \{[2,3,\ldots,n+1],\ [3,4,\ldots,n+2],\ [2,3,\ldots,n,n+2]\}\right)$$
and consider the specialization 
 obtained by setting all variables in $\mathscr{T}$ to zero.
 Notice that $\mathscr{P} \subseteq \mathscr{T}$, that is, after specializing, all Pl\"ucker relations become zero. Indeed, any pair among the minors 
$[2,3,\ldots,n+1]$, $[3,4,\ldots,n+2]$, and $[2,3,\ldots,n,n+2]$ share $n-1$ indices; hence, by Remark~\ref{nonTrivialPL}, no nontrivial Plücker relation can be formed from the product of two of them. Therefore, after specialization,

$$\dfrac{\calP_\calF}{\mathscr{P}+(f_{\Lap})+\mathscr{T}} =\dfrac{\mathbb{K}\bigl[T_{2,3,\ldots,n+1},\ T_{3,4,\ldots,n+2},\  T_{2,3,\ldots,n,n+2}\bigr]}{(\overline{f})},$$

where $$\overline{f}=T_{2,3,\ldots,n+1}^{\,n-1}\,T_{3,4,\ldots,n+2}+
T_{2,3,\ldots,n,n+2}^{\,n}.$$
Note that $\overline f$ is a primitive linear polynomial in $T_{3,4,\ldots,n+2}$ over the UFD $A:=\mathbb{K}\bigl[T_{2,3,\ldots,n+1},\,T_{2,3,\ldots,n,n+2}\bigr]$ whose coefficients are coprime; hence, $\overline f$ is irreducible over $A\bigl[T_{3,4,\ldots,n+2}\bigr]$. \par 

The irreducibility of $\overline f$ over $A\bigl[T_{3,4,\ldots,n+2}\bigr]$ implies the irreducibility of $f_{\Lap}$ over $\calF[0]$. In fact, suppose that $f_{\Lap} - gh \in \mathscr{P}$ for some polynomials $g$ and $h$ in $\calP_\calF$. Since $f_{\Lap}$ is homogeneous, we may choose $g$ and $h$ to be homogeneous in $\calF[0]$. Because all the Pl\"ucker relations become $0$ after specialization, we get

$$\overline{f}=\overline{g} \overline{h}.$$

 Since $\overline{f}$ is an irreducible homogeneous polynomial, one of the factors, say $\overline{g}$, must be a unit (of degree zero). Therefore $g$ is a homogeneous polynomial of degree zero modulo $\mathscr{P}$, so a unit modulo $\mathscr{P}$. This completes the proof.

\end{proof}

\begin{Remark}
    The proof of irreducibility is independent of the base field and the polynomial $f_{\Lap}$ is geometrically irreducible.
\end{Remark}

\begin{Corollary} \label{cor:alginvariant r=1}The $a$-invariant of $\calF[1]$ is $-2$, and its regularity is $2n-2$ which coincides with the reduction number of $I_n(H[1]).$

\end{Corollary} 

    \begin{proof} Let $X$ be a generic matrix of size $n\times (n+2)$. By \cite[Corollary 1.4]{BH} we know that

$$a(\mathcal{F}(I_n(X)))=-(n+2).$$

It follows by the isomorphism in \cite[Proposition 4.1]{CMRS}, that $a(\calF[0])=-(n+2).$

Since $f_{\Lap}$ is a regular homogeneous element of degree $\mathrm{deg}(f_{\Lap})=n$ over $\calF[0]$, by  \cite[Corollary 3.6.14]{BrunsH} we have

$$a(\calF [1])=a(\calF[0])+\mathrm{deg}(f_{\Lap})=-n-2+n=-2.$$

The $a$-invariant equals the degree of the Hilbert series of the corresponding ring viewed as a rational function. In particular, since $\calF[1]$ is Cohen-Macaulay, we have $$a(\calF [1])=\reg(\calF[1])-\dim(\calF[1]).$$ Moreover, by \cite[Proposition 1.85]{Vas} we have that $\reg(\calF[1])={\rm r}(I_n(H[1]))$. Therefore 
$$-2=\reg(\calF[1])-\dim(\calF[1])={\rm r}(I_n(H[1]))-\ell (I_n(H[1]))={\rm r}(I_n(H[1]))-(2n).$$
 Hence $\reg(\calF[1])={\rm r}(I_n(H[1]))=2n-2.$
    \end{proof}
  \medskip
 \begin{Corollary}\label{cor: mult}
     The multiplicity of $\calF[1]$ is $\frac{n}{n+1} {2n \choose n}.$
 \end{Corollary}
\begin{proof}
   The isomorphism between the Grassmannian $\Gr(n,n+2)$ and $\calF[0]$ implies that they have the same multiplicity. By duality property of Grassmannians, $\Gr(n,n+2)\cong \Gr(2,n+2)$, and the multiplicity of $\calF[0]$ is the multiplicity of $\Gr(2,n+2)$ which is the well-known Catalan number, $\frac{1}{n+1}{2n\choose n}$, (see \cite[Example 14.7.11]{Fulton}). Since $f_{\Lap}$ is a homogeneous regular element of degree $n$ in $\calF[0]$, and $\calF[1] \isom \calF[0]/(f_{\Lap})$, then
   $$
e(\calF[1]) = \deg(f_{\Lap}) \cdot e(\calF[0]) = n \cdot \frac{1}{n+1} \binom{2n}{n}.
\eqno\qedhere$$
\end{proof}

We conclude this section by observing that $I_n(H[1])$ is of fiber type. In particular, this provides explicit defining equations for  $\calR[1]$. In doing so, we also prove that $I_n(H[1])$ has linear powers, that is, all its powers have linear resolution. It is known that if the Rees algebra of an ideal $I$ is Koszul, then $I$ has linear powers \cite[Corollary~3.6]{Blum}. This applies in particular to non-degenerate Hankel determinantal ideals, whose Rees algebras were shown to be Koszul by Conca \cite[Theorem~4.7]{Co}.
The next result shows that $I_n(H[1])$ also has linear powers, although its Rees algebra is not Koszul.

\begin{Proposition}\label{prop:r=1 fiber type}
    The Rees algebra $\calR[1]$ is of fiber type and $I_n(H[1])$ has linear powers. In particular, its defining ideal $\calJ[1]$ is generated by $f_{\Lap}$ together with the Eagon–Northcott relations \eqref{eq:ENeq} and the Plücker relations \eqref{eq:general-plucker}.
\end{Proposition}

\begin{proof}
    The fiber type property and linear power property of $I_n(H[1])$ follow directly from \cite[Theorem 3.7]{BCV} and \cite[Proposition~2.4]{CMRS}. This implies that 
$\calJ[1]$ is generated by the defining ideal of the symmetric algebra together with the defining equations of the special fiber. Since $I_n(H[1])$ has expected codimension by \cite[Proposition~2.4(i)]{CMRS}, it is well-known that the Eagon-Northcott complex resolves $I_n(H[1])$ \cite[Theorem 2]{EN}. Combining this result with Theorem \ref{thm:eqSFr=1}, we conclude that $f_{\Lap}$, together with the Eagon–Northcott and Plücker relations, generates $\calJ[1]$, as claimed.
\end{proof}

As a consequence, we can also recursively compute the Betti numbers of $I_n(H[1])^k$ for all $k$.

\begin{Corollary} For all $k$, we have 
    $$\beta_i^{S_1}(I_n(H[1])^k)=\beta^{S}_i(I_n(\calH)^k)-\binom{2n}{i} \beta_{2n}^{S}(I_n(\calH)^k).$$
\end{Corollary}
\begin{proof}
   The variable  $x_{2n+1}$ is an almost $S/(I_n(\calH)^k)$-regular element for all $k$ if and only if $x_{2n+1}\not \in P$ for any prime ideal $P\in \cup_k \;\mathrm{Ass}(S/(I_n(\calH)^k))$ with $P\neq \mathfrak{m}=(x_1, \ldots, x_{2n+1})$ where $\mathfrak{m}=(x_1,\ldots,x_{2n+1})$ denotes the homogeneous maximal ideal. The result follows from \cite[Lemma 8.7.8]{BCRV} and \cite[Corollary~3.18]{Co}.
\end{proof}

\section{Further questions}
We conclude by discussing some interesting questions that emerge from our work.\par
A natural next step would be to analyze the intermediate cases and gain further insight into the structure of the special fibers of $n$-determinantal ideals of codimension $3$.
\begin{Problem}
    Find the defining equations of $\calF[r]$ for $r=2,\ldots, n-2$.
\end{Problem}
 Based on Remark~\ref{rmk:chainF} and on the structure of the matrices defining our equations (see also Example \ref{ex:f-relation}), it is natural to conjecture that $\calK[r]$ is minimally generated by Pl\"ucker relations and additional relations of degree $n-r+1$ coming from the vanishing of certain minors with similar structure to the ones presented in this paper. The main challenge, however, lies in proving that these relations indeed generate the entire ideal $\calK[r]$.
 \medskip

In this context, it would be interesting and useful to understand the combinatorics underlying the relations described in this paper.
\begin{Problem}
   Describe $\Lap_{\aa}$  and $f_{\Lap}$ in terms of their underlying posets.
\end{Problem}

It is well-known that the Plücker relations \eqref{eq:minors relation} can be read directly from the poset of maximal minors: each incomparable pair in the poset corresponds to a Plücker relation, and the remaining two terms in such a relation are products of comparable minors located symmetrically in upper and lower rank positions with respect to the incomparable one. In particular, all the terms in a Pl\"ucker relation have the same \textit{total rank} in the poset, that is, the sum of the indices appearing in each product is constant. We are interested in finding a similar characterization for the other classes of relations studied in this work. It turns out that, even in these cases, the total rank of all the terms in a relation is constant and can be explicitly computed. For $f_{\Lap}$ this constant is $ \frac{n(n+1)(n+2)}{2}$, while for $\Lap_\aa$, it is $\sum_{a\in\mathbf a}a + \frac{(n+2)(n+3)}{2} + 2$. Moreover, using the isomorphism in Lemma \ref{lem:isom}, one can determine precisely which pairs of minors in the poset give rise to a Laplace relation \eqref{eq:lap} by passing to the corresponding poset of $2$-minors and using the term order introduced in \cite{RS}. However, a complete explicit description of the terms appearing in each of these relations in terms of the poset remains unknown.
\medskip

 Another interesting question regards the fiber type and linear powers properties  of these ideals.

 \begin{Problem}
     Is $\calR[r]$ of fiber type for $r\geq 2$? Does $I_n(H[r])$ have linear powers for $r\geq 2$?
 \end{Problem}

    Note that \cite[Theorem 3.7]{BCV} does not apply to the case $H[r]$ with $r \geq 2$, as the hypothesis $\Ht(I_j(H[r]))\geq \min\{2(n+1-j)+1, 2n+1-r\}$ is not satisfied. Similarly, an inductive argument via \cite[Lemma 2.4]{BCV} cannot be pursued, because the finite-length assumption fails already for $k=2$. Still, computational experiments suggest that $\calR[r]$ is of fiber type for any $r$, and the ideal has linear powers (verified for small values of $k$).

\subsection*{Acknowledgments.} Our work started at the 2024 workshop \textit{Women in Commutative
Algebra III} hosted by Casa Matemática Oaxaca. We thank the organizers of this workshop for bringing our team together. We acknowledge the excellent working conditions
provided by CMO-BIRS.\\
The authors thank Yevgeniya Jonah Tarasova for discussions
concerning the content of the paper. They also thank Alessio Sammartano and Matteo Varbaro for comments on a previous draft of this work. 
KS is partially supported by SERB(ANRF) POWER grant SPG/2022/002099, Govt.
of India. LS is partially supported by SNSF grant TMPFP2\_217223.

 \bibliographystyle{amsalpha}
\bibliography{ref}
\end{document}